\documentclass[12pt]{amsart}
\usepackage{latexsym, amssymb, amsmath,mathtools}

\newtheorem{thm}{Theorem}[section]
\newtheorem{lem}[thm]{Lemma}
\newtheorem{cor}[thm]{Corollary}
\newtheorem{prop}[thm]{Proposition}
\newtheorem{rem}[thm]{Remark}
\theoremstyle{definition}
\newtheorem{defn}[thm]{Definition}

\newtheorem{example}{Example}

\newcommand{\Ric}{{\rm Ric}}

\begin{document}

\title[Generalized Yamabe solitons]
{Classification of generalized Yamabe solitons in Euclidean spaces}


\author{Shunya Fujii}
\address{Department of Mathematics,
 Shimane University, Nishikawatsu 1060 Matsue, 690-8504, Japan.}
\curraddr{}
\email{shunyaf3@gmail.com}

\author{Shun Maeta}
\address{Department of Mathematics,
 Shimane University, Nishikawatsu 1060 Matsue, 690-8504, Japan.}
\curraddr{}
\email{shun.maeta@gmail.com~{\em or}~maeta@riko.shimane-u.ac.jp}
\thanks{The second author is partially supported by the Grant-in-Aid for Young Scientists, No.19K14534, Japan Society for the Promotion of Science.}


\date{}

\dedicatory{}

\commby{}

\begin{abstract}
In this paper, we consider generalized Yamabe solitons which include many notions, such as Yamabe solitons, almost Yamabe solitons, $h$-almost Yamabe solitons, gradient $k$-Yamabe solitons and conformal gradient solitons.
We completely classify the generalized Yamabe solitons on hypersurfaces in Euclidean spaces arisen from the position vector field.
\end{abstract}

\maketitle


\bibliographystyle{amsplain}

\section{Introduction}\label{intro}

An $n$-dimensional Riemannian manifold $(M,g)$ is called a {\it Yamabe soliton}, if there exist a complete vector field $v$ and $\rho\in \mathbb{R}$ such that
\begin{equation}\label{YS}
(R-\rho)g=\frac{1}{2}\mathcal{L}_{v}g,
\end{equation}
where $R$ is the scalar curvature of $M$ and $\mathcal{L}_vg$ is the Lie derivative of $g$.
If $v$ is the gradient of some smooth function $f$ on $M$, then $(M,g,f)$ is called a gradient Yamabe soliton. If $f$ is constant, then $(M,g,f)$ is called trivial.

Yamabe solitons are special solutions of the Yamabe flow introduced by R. Hamilton \cite{Hamilton89}.
In the last decade, Yamabe solitons have developed rapidly.
To understand the Yamabe soliton, many generalizations of it have been introduced:
\begin{enumerate}
\item
Almost Yamabe solitons \cite{BB13}:

For a complete vector field $v$  and a smooth function $\rho$ on $M$,
$$(R-\rho)g=\frac{1}{2}\mathcal{L}_{v}g.$$

\item
Gradient $k$-Yamabe solitons \cite{CMM12}:

For a smooth function $f$ on $M$ and $\rho\in\mathbb{R}$,
$$2(n-1)(\sigma_k-\rho)g=\nabla\nabla f,$$
where $\sigma_k$ denotes the $\sigma_k$-curvature of $g$, that is, 
$$\sigma_k=\sigma_k(g^{-1}A)=\sum_{i_1<\cdots<i_{k}}\mu_{i_1}\cdots\mu_{i_k}~~~~~(\text{for}~1\leq k\leq n),$$
where $A=\frac{1}{n-2}(\Ric-\frac{1}{2(n-1)}Rg)$ is the Schouten tensor and $\mu_1,\cdots,\mu_n$ are the eigenvalues of the symmetric endomorphism $g^{-1}A$. Here, $\Ric$ is the Ricci tensor of $M$. We remark that the $1$-Yamabe soliton is the Yamabe soliton.\\

\item
$h$-almost Yamabe solitons \cite{Zeng20}:

For a complete vector field $v$ and smooth functions $\rho$ and $h~(h>0~\text{or}~h<0)$ on $M$,
$$(R-\rho)g=\frac{h}{2}\mathcal{L}_vg.$$

\item
Conformal gradient solitons \cite{CMM12}:

For smooth functions $f$ and $\varphi$ on $M$,
$$\varphi g=\nabla\nabla f.$$
\end{enumerate}

Many examples of these solitons are known.
A warped product manifold $(\mathbb{R}\times_{\cosh t}\mathbb{S}^n,~ dt^2+\cosh^2t g_{\mathbb{S}^n},~ f=\sinh t,~ \rho=\sinh t+n)$ is a gradient almost Yamabe soliton, where $(\mathbb{S}^n,g_{\mathbb{S}^n})$ is the standard sphere (cf. \cite{BB13}).
On the cylinder $\mathbb{S}^n\times\mathbb{R}$, gradient $k$-Yamabe solitons are constructed (cf. \cite{BHS18}).
For a nonzero real number $m$ and a positive constant $\beta$, $(\mathbb{R}^n,g_{can},f=-m\log (|x|^2+\beta),\rho=\frac{2m}{|x|^2+\beta})$ is a gradient $(-\frac{m}{|x|^2+\beta})$-almost Yamabe soliton, where $(\mathbb{R}^n,g_{can})$ is the Euclidean space (cf.~\cite{Zeng20}).

Even though there are many examples of these solitons, some classifications are known (cf. \cite{BB13}, \cite{BHS18}, \cite{CSZ12}, \cite{CMM12}, \cite{2}, \cite{3}, \cite{4}, \cite{Maeta19}, \cite{Maeta20}, \cite{SM19} and \cite{Zeng20}).
In particular, there are many results for Yamabe solitons.
Any compact Yamabe soliton has constant scalar curvature (cf. \cite{CD08}, \cite{CLN06}  and \cite{3}). P. Daskalopoulos and N. Sesum \cite{2} showed that any locally conformally flat complete gradient Yamabe soliton with positive sectional curvature has to be rotationally symmetric.  
G. Catino, C. Mantegazza and L. Mazzieri \cite{CMM12} classified nontrivial complete gradient Yamabe solitons with nonnegative Ricci tensor.
H.-D. Cao, X. Sun and Y. Zhang \cite{CSZ12} gave a useful classification.
Recently, the second author \cite{Maeta20} classified 3-dimensional complete gradient Yamabe solitons with divergence-free Cotton tensor. 

To consider all the generalized Yamabe solitons mentioned above, we consider the following.
\begin{defn}
A Riemannian manifold $(M,g)$ is called a {\em conformal soliton} if there exists a complete vector field $v$ such that 
\begin{equation}\label{CS}
\varphi g=\frac{1}{2}\mathcal{L}_vg,
\end{equation} 
for a smooth function $\varphi:M\rightarrow \mathbb{R}.$ We denote the conformal soliton by $(M,g,v,\varphi)$.
If $v\equiv0$, then $M$ is called trivial.
\end{defn}
\begin{rem}
Conformal solitons include Yamabe solitons, almost Yamabe solitons, gradient $k$-Yamabe solitons, $h$-almost Yamabe solitons and conformal gradient solitons.
 Therefore, {\bf all the results in this paper can be applied to all these solitons.}

$v$ is called a conformal vector field.
\end{rem}
We can construct many examples of conformal solitons. For example, for a smooth function $f=\log(e^x+e^y+1)$ on $\mathbb{R}^2$ with coordinate system $\{x,y\}$, $g=\nabla\nabla f$ is a Riemannian metric on $\mathbb{R}^2$. Thus $(\mathbb{R}^2,g,f,\varphi=1)$ is a conformal (gradient) soliton. 

To understand conformal solitons, we consider them as hypersurfaces of some Riemannian manifold. 
As the metric of $M$ is induced from the ambient space, it seems natural to take a soliton vector field $v$ from the ambient space. Let $V$ be some vector field of the ambient space. Then we can decompose $V$ as the tangential component $V^T$ and the normal component $V^\perp$. Therefore, if the ambient space has a (natural) vector field $V$, it is natural to take $v=V^T$. The Euclidean space $\mathbb{E}^{n+1}$ is the most basic and natural one, because it has the position vector field $V$. 
Some solitons on hypersurfaces in Euclidean spaces or more general manifolds have been studied (cf. \cite{1}, \cite{CD18}, \cite{CK09}, \cite{CK11}, \cite{CK12} and \cite{SM19}).

In this paper, we completely classify conformal solitons on a hypersurface in the Euclidean space $\mathbb{E}^{n+1}$ arisen from the position vector field. 

\begin{thm}\label{main}
Any conformal soliton $(M,g,V^T,\varphi)$ on a hypersurface in the Euclidean space $\mathbb{E}^{n+1}$ is contained in a hyperplane, a conic hypersurface or a hypersphere.
\end{thm}
A submanifold $M^n$ in the Euclidean space $\mathbb{E}^m$ is called a {\it conic submanifold} if it is an open portion of a cone with vertex at the origin (cf. \cite{Chen16}). Here we remark that by the definition, $n$-dimensional planes through the origin are included in $n$-dimensional conic submanifolds. 
 
The following fact is used later.

\begin{prop}[\cite{Chen16}]\label{Chenlem}
Let $f:M\rightarrow \mathbb{E}^m$ be an isometric immersion of an $n$-dimensional Riemannian manifold into the $m$-dimensional Euclidean space $\mathbb{E}^m$. Then $V=V^T$ holds identically if and only if $M$ is a conic submanifold.
\end{prop}

\section{Preliminaries}\label{Pre} 

Let $(N,\tilde{g})$ be an $m$-dimensional Riemannian manifold and $(M,g)$ be an $n$-dimensional submanifold in $(N,\tilde{g})$.
All manifolds in this paper are assumed to be smooth, orientable and connected.
 We denote Levi-Civita connections on $(M,g)$ and $(N,\tilde{g})$ by $\nabla$ and $\tilde{\nabla}$, respectively.
The Lie derivative of $g$ is defined by 
$$\mathcal{L}_Xg(Y,Z)=X(g(Y,Z))-g([X,Y],Z)-g(Y,[X,Z]),$$
for any vector fields $X,Y, Z$ on $M$.

For any vector fields $X,Y$ tangent to $M$ and $\eta$ normal to $M$, the formula of Gauss is given by
\begin{equation*}
{\tilde{\nabla}}_XY={\nabla}_XY+h(X,Y),
\end{equation*}
where ${\nabla}_XY$ and $h(X,Y)$ are the tangential and the normal components of ${\tilde{\nabla}}_XY$.
The formula of Weingarten is given by
\begin{equation*}
{\tilde{\nabla}}_X\eta=-A_{\eta}(X)+D_X\eta,
\end{equation*}
where $-A_{\eta}(X)$ and $D_X\eta$ are the tangential and the normal components of ${\tilde{\nabla}}_X\eta$.
$A_{\eta}(X)$ and $h(X,Y)$ are related by
\begin{equation*}
g(A_{\eta}(X),Y)=\tilde{g}(h(X,Y),\eta).
\end{equation*}
The mean curvature vector $H$ of $M$ in $N$ is given by
\begin{equation*}
\displaystyle H=\frac{1}{n}~\text{trace}~h.
\end{equation*}
For any vector fields $X,Y,Z,W$ tangent to $M$, the equation of Gauss is given by
\begin{equation*}
\begin{tabular}{ll}
$\tilde{g}(\tilde{Rm}(X,Y)Z,W)=$ & $g(Rm(X,Y)Z,W)$ \vspace{0.3pc}\\
~ & $+\tilde{g}(h(X,Z),h(Y,W))$ \vspace{0.3pc}\\
~ & $-\tilde{g}(h(X,W),h(Y,Z)),$
\end{tabular}
\end{equation*}
where $Rm$ and $\tilde{Rm}$ are Riemannian curvature tensors of $M$ and $N$, respectively.
The equation of Codazzi is given by
\begin{equation*}
(\tilde{Rm}(X,Y)Z)^{\perp}=({\bar{\nabla}}_Xh)(Y,Z)-({\bar{\nabla}}_Yh)(X,Z),
\end{equation*}
where $(\tilde{Rm}(X,Y)Z)^{\perp}$ is the normal component of $\tilde{Rm}(X,Y)Z$ and ${\bar{\nabla}}_Xh$ is defined by
\begin{equation*}
({\bar{\nabla}}_Xh)(Y,Z)=D_Xh(Y,Z)-h({\nabla}_XY,Z)-h(Y,{\nabla}_XZ).
\end{equation*}
If $N$ is a space of constant curvature, then the equation of Codazzi reduces to
\begin{equation*}
0=({\bar{\nabla}}_Xh)(Y,Z)-({\bar{\nabla}}_Yh)(X,Z).
\end{equation*}

\section{Conformal solitons with a concurrent vector field}\label{CSwcv}

The position vector field $V$ on Euclidean spaces satisfies
$$\nabla_XV=X,$$
for any vector field $X$.
We consider one of the generalizations of the position vector field, namely, a concurrent vector field.
\begin{defn}
A vector field $V$ on $M$ is called a concurrent vector field if it satisfies
$$\nabla_XV=X,$$
for any vector field $X$ on $M$.
\end{defn}
There are several studies of concurrent vector fields (see for example \cite{BY74}, \cite{5} and \cite{6}).

In this section, we consider a conformal soliton with a concurrent vector field.
 
Firstly, we show a useful formula for study of conformal solitons.

\begin{lem}
Let $(M,g,f,\varphi)$ be a conformal gradient soliton. Then, we have
\begin{equation}\label{f1}
(n-1)\Delta{\varphi}+\frac{1}{2}~g(\nabla R,\nabla f)+R \varphi =0.
\end{equation}
\end{lem}

\begin{proof}
Since
\begin{equation*}
\Delta {\nabla}_if={\nabla}_i\Delta f+R_{ij}{\nabla}_jf,
\end{equation*}
\begin{equation*}
\Delta {\nabla}_if={\nabla}_k{\nabla}_k{\nabla}_if={\nabla}_k(\varphi g_{ki})={\nabla}_i\varphi,
\end{equation*}
and
\begin{equation*}
{\nabla}_i\Delta f={\nabla}_i(n\varphi)=n{\nabla}_i\varphi,
\end{equation*}
we have
\begin{equation}\label{f1.1}
(n-1){\nabla}_i\varphi+R_{ij}{\nabla}_jf=0,
\end{equation}
where $R_{ij}$ is the Ricci tensor of $M$.
By applying ${\nabla}_l$ to the both sides of $(\ref{f1.1})$, we obtain
\begin{equation}\label{f1.2}
(n-1){\nabla}_l{\nabla}_i\varphi+{\nabla}_lR_{ij} \cdot {\nabla}_jf+R_{ij}{\nabla}_l{\nabla}_jf=0.
\end{equation}
Taking the trace, we obtain $(\ref{f1})$.
\end{proof}

\begin{prop}\label{AYSwithC}
Any conformal soliton $(M,g,v,\varphi)$ which has a concurrent vector field $v$ is a conformal gradient soliton with 
$\varphi =1$.
\end{prop}

\begin{proof}
Since $v$ is a concurrent vector field, we have
\begin{equation}\label{c.0}
g(v,X) =g(v,{\nabla}_Xv)=X(\frac{1}{2}g(v,v)),
\end{equation}
and
\begin{equation}\label{c.1}
\begin{tabular}{ll}
$\mathcal{L}_vg(X,Y)$ & = $vg(X,Y) - g([v,X],Y) - g(X,[v,Y])$ \vspace{0.3pc}\\
~& $= vg(X,Y) - vg(X,Y) + g({\nabla}_Xv,Y) + g(X,{\nabla}_Yv)$ \vspace{0.3pc}\\
~& $= 2g(X,Y),$
\end{tabular}
\end{equation}
for any vector fields $X$, $Y$ on $M$.
By putting $f = \frac{1}{2}g(v,v)$ on the equation $(\ref{c.0})$, we obtain $ v = \nabla f $.
Substituting $(\ref{c.1})$ into $(\ref{CS})$, we have
\begin{equation*}
\varphi = 1.
\end{equation*}
\end{proof}

From Proposition~\ref{AYSwithC}, the equation of the conformal soliton with a concurrent vector field is as follows:
$$g=\nabla\nabla f.$$
Therefore $g$ should be a Hessian metric. The Hessian metric is an important notion on Geometry and Physics (cf. \cite{AA14}, \cite{CY80}, \cite{MM17} and \cite{SY97}).

\begin{example}
A Hessian manifold $(M,g)$ is a conformal gradient soliton with $\varphi=1$.
\end{example}

If $M$ is compact, then there exist no non trivial conformal solitons with a concurrent vector field.
\begin{cor}
There exists no compact conformal soliton such that the conformal vector field is a concurrent vector field.
\end{cor}

\begin{proof}
By Proposition $\ref{AYSwithC}$ and $(\ref{CS})$, we have $\Delta f=n$.
By applying maximum principle, we get $f$ is constant, which cannot happen.
\end{proof}
\section{Conformal solitons on submanifolds}\label{SubCS}

In this section, we assume that $(N,\tilde{g})$ is a Riemannian manifold endowed with a concurrent vector field $V$ and $(M,g)$ is a submanifold in $(N,\tilde{g})$.
$V^T$ and $V^{\perp}$ denote the tangential and the normal components of $V$, respectively.

Firstly, we show the following lemma which will be used later for the purpose of classification of the conformal solitons. 

\begin{lem}\label{NSAYS}
Any conformal soliton $(M,g,V^T,\varphi)$ on a submanifold $M$ in $N$ satisfies
\begin{equation}\label{ENSAYS}
(\varphi-1)g(X,Y)=g(A_{V^{\perp}}(X),Y),
\end{equation}
for any vector fields $X, Y$ on $M$.
\end{lem}

\begin{proof}
From the definition of the Lie derivative, we have
\begin{equation}\label{s.3}
\begin{tabular}{ll}
$(\mathcal{L}_{V^T}g)(X,Y)$ & = $V^Tg(X,Y) - g(\nabla_{V^T}X-\nabla_X{V^T},Y) - g(X,\nabla_{V^T}Y-\nabla_Y{V^T})$\vspace{0.5pc}\\
~ & $=g(\nabla_X{V^T},Y) + g(X,\nabla_Y{V^T})$\vspace{0.5pc}\\
~ & $=\tilde g(\tilde{\nabla}_X{V}-\tilde{\nabla}_X{V^{\perp}},Y) + \tilde g(X,\tilde{\nabla}_Y{V}-\tilde{\nabla}_Y{V^{\perp}})$\vspace{0.5pc}\\
~ & $=2g(X,Y)+2g(A_{V^{\perp}}(X),Y)$,
\end{tabular}
\end{equation}
for any vector fields $X, Y$ on $M$.
Combining $(\ref{s.3})$ with $(\ref{CS})$, we obtain $(\ref{ENSAYS})$.
\end{proof}

\begin{prop}\label{CS is GCS}
Any conformal soliton $(M,g,V^T,\varphi)$ on a submanifold $M$ in $N$  is a conformal gradient soliton. 
\end{prop}

\begin{proof}
Set
\begin{equation*}
\displaystyle f=\frac{1}{2}~\tilde{g}(V,V).
\end{equation*}
For any vector field $X$ on $M$, we obtain
$$g(V^T,X)=\tilde g(V,X)=\tilde g (V,\tilde \nabla_XV)=X(\frac{1}{2}\tilde g(V,V))=Xf=g(\nabla f, X).$$
\end{proof}

\begin{prop}\label{MCS}
If a conformal soliton $(M,g,V^T,\varphi)$ on a submanifold $M$ in $N$ is minimal, then $\varphi=1$.
\end{prop}

\begin{proof}
Let $\{e_1,  \cdots , e_n\}$ be an orthonormal frame on $M$. By Lemma~$\ref{NSAYS}$, we have
\begin{equation*}
(\varphi-1)g_{ij}=g(A_{V^{\perp}}(e_i),e_j)=\tilde{g}(h(e_i,e_j),V^{\perp}).
\end{equation*}
Since $M$ is minimal and taking the trace, we obtain
\begin{equation*}
n(\varphi-1)=n\tilde g(H,V^{\perp})=0. 
\end{equation*}
Therefore, we conclude that
\begin{equation*}
\varphi=1.
\end{equation*}
\end{proof}

\begin{cor}
There exists no compact conformal soliton on a minimal submanifold in $N$ such that the conformal vector field is $V^T$. 
\end{cor}

\begin{proof}
By Proposition $\ref{MCS}$ and $(\ref{CS})$, we have $\Delta f=n$.
By applying maximum principle, we get $f$ is constant, which cannot happen.
\end{proof}

\section{Proof of Theorem~$\ref{main}$}\label{Classification}

We hereafter denote $V$ by the position vector field of $\mathbb{E}^{n+1}$.
In this section, we give the proof of Theorem~$\ref{main}$ as follows.

\begin{thm}
Any conformal soliton $(M,g,V^T,\varphi)$ on a hypersurface in the Euclidean space $\mathbb{E}^{n+1}$ is contained in a hyperplane, a conic hypersurface or a hypersphere.
\end{thm}

\begin{proof}
Let $\alpha$ be a mean curvature and $\lambda$ be a support function of $M$,  i.e., $H=\alpha N$ and $\lambda =\tilde{g}(N,V)$ with a unit normal vector field $N$. Set $U_0=\{x\in M| \lambda=0\}$.
From Lemma~$\ref{NSAYS}$, we have
\begin{equation*}
(\varphi-1)g_{ij}=\tilde{g}(h(e_i,e_j),V^{\perp})=\tilde{g}({\kappa}_i g_{ij} N,V)={\kappa}_i g_{ij} \lambda ,
\end{equation*}
where $A_N(e_i)={\kappa}_ie_i, ~ (i=1,\cdots ,n)$. 
Hence we have
\begin{equation}\label{e.1}
\varphi-1=\lambda {\kappa}_i.
\end{equation}

 Let $f$ be a smooth function on $M$ defined by 
 
 \begin{equation*}
f(x)\coloneqq \prod_{i=1}^n \kappa_i(x) \hspace{0.2in} x\in M.
\end{equation*} 

\noindent
\underline{Case 1. $U_0 = \emptyset$}:
By taking the summation, we obtain
\begin{equation}\label{e.2}
\varphi-1=\lambda \alpha.
\end{equation} 
Comparing ~$(\ref{e.1})$ and ~$(\ref{e.2})$, we have
\begin{equation*}
{\kappa}_i = \alpha.
\end{equation*}
Thus $M$ is totally umbilical with $A_N(e_i)=\alpha e_i$ and $h$ satisfies $h(X,Y)=\alpha g(X,Y) N$.
Since $N$ is a unit normal vector field, we have
\begin{equation*}
0={\tilde{\nabla}}_X(\tilde{g}(N,N))=2\tilde{g}({\tilde{\nabla}}_XN,N)=2\tilde{g}(D_XN,N).
\end{equation*}
Therefore, $D_XN=0$.
Hence we obtain
\begin{equation*}
\begin{tabular}{rl}
$({\bar{\nabla}}_Xh)(Y,Z)=$ & $D_Xh(Y,Z)-h({\nabla}_XY,Z)-h(Y,{\nabla}_XZ)$ \vspace{0.5pc} \\
~$=$ & $X(\alpha ) g(Y,Z) N,$
\end{tabular}
\end{equation*}
for any vector fields $X, Y, Z$ on $M$.
From the equation of Codazzi, we have
\begin{equation*}
X(\alpha)Y=Y(\alpha)X.
\end{equation*}
Since we can assume that $X$ and $Y$ are linearly independent, we conclude that $\alpha$ and $f$ are constant respectively.

If $\alpha=0$, then by ${\tilde{\nabla}}_XN=0,$ $N$, restricted to $M$, is a constant vector field in $\mathbb{E}^{n+1}$
and we have
\begin{equation*}
{\tilde{\nabla}}_X(\tilde{g}(V,N))=\tilde{g}({\tilde{\nabla}}_XV,N)+\tilde{g}(V,{\tilde{\nabla}}_XN)=\tilde{g}(X,N)=0.
\end{equation*}
This shows that $\lambda=\tilde{g}(V,N)$ is constant when $V$ and $N$ are restricted to $M$. 
Therefore, $M$ is contained in a hyperplane normal to $N$ which does not through the origin and $\varphi=1$.

If $\alpha\not=0$, then we have
\begin{equation*}
{\tilde{\nabla}}_X(V+{\alpha}^{-1}N)=X+{\alpha}^{-1} {\tilde{\nabla}}_XN=X+{\alpha}^{-1} (-A_N(X))=0.
\end{equation*}
This shows that the vector field $V+{\alpha}^{-1}N$, restricted to $M$, is a constant one in $\mathbb{E}^{n+1}.$ 
Therefore, $M$ is contained in a hypersphere.

\noindent
\underline{Case 2. $U_0=M$}:

We have $V=V^T$. By Proposition~\ref{Chenlem}, we obtain that $M$ is contained in a conic hypersurface.

\noindent
\underline{Case 3. Others}:

Take $p\in M\backslash U_0$, that is, $\lambda\not=0$ on some open set $\Omega\ni p$.
By the same argument as in Case 1, we have $\Omega$ is an open portion of a hyperplane or a hypersphere.

We consider the case that $\Omega$ is an open portion of a hyperplane.
Without loss of generality, we can take $\Omega$ as the maximum connected component which is an open set including $p$ on $M\backslash U_0$.
On $\Omega$, $\lambda=\tilde g(V,N)(\not=0)$ is constant, say $\lambda_\Omega$.
Since $M$ is connected, if $\Omega$ is closed, then $\Omega=M$, which is a contradiction.
If $M$ is not closed, then we have
$\partial \Omega\not=\emptyset$ and $\partial \Omega\cap\Omega=\emptyset.$
Take $q\in \partial \Omega$. Since $\lambda$ is continuous, we have $\lambda(q)=\lambda_\Omega$. Thus we can take an open neighborhood $U_q$ of $q$ such that $\lambda\not=0$ on $U_q$. Since $\Omega$ is the maximum connected component, we have a contradiction.
Hence, we have that $\Omega$ is an open portion of a hypersphere. 

If Int$U_0 \neq \emptyset$, for any $x\in$ Int$U_0$, we can take some open set $U_x\ni x$ which is included in $U_0$.
Let $\{ e_i\}$ be an orthonormal frame defined on $U_x$ which satisfies

\begin{equation}
A_N(e_i)(x)=\kappa_i(x)e_i(x),
\end{equation} 

and

\begin{equation}
\nabla_{e_i}e_j(x)=0, \hspace{0.2in} 1\leq i,j\leq n.
\end{equation} 

On $U_0$, we have $V=V^T$. Since $V$ is a concurrent vector field,

\begin{equation*}
X =\tilde{\nabla}_XV^T = \nabla_XV^T+h(X,V^T),
\end{equation*}

on Int$U_0$.

So we obtain

\begin{equation}\label{c}
\nabla_XV^T=X,
\end{equation} 

and

\begin{equation}\label{h}
h(X,V^T)=0.
\end{equation}

From $(\ref{c})$,

\begin{equation*}
\begin{tabular}{rl}
$e_i(x)=$ & $\nabla_{e_i}V^T(x)$ \vspace{0.5pc} \\
~$=$ & $\nabla_{e_i}(V^T)^je_j(x)$ \vspace{0.5pc} \\
~$=$ & $e_i(V^T)^j(x)e_j(x)+(V^T)^j(x)\nabla_{e_i}e_j(x)$.
\end{tabular}
\end{equation*}

Since $\nabla_{e_i}e_j(x)=0$,

\begin{center}
$e_i(x)=e_i(V^T)^j(x)e_j(x)$.
\end{center}

Therefore,

\begin{equation}\label{c2}
e_i(V^T)^j(x)=\delta_{ij}.
\end{equation} 

From $(\ref{h})$,

\begin{equation}\label{h2}
\begin{tabular}{rl}
$0=$ & $\tilde{g}(h(e_i,V^T),N)$ \vspace{0.5pc} \\
~$=$ & $g(A_N(e_i),\displaystyle \sum_{j=1}^n(V^T)^je_j)$ \vspace{0.5pc} \\
~$=$ & $\displaystyle \sum_{j=1}^n (V^T)^j\tau_{ij}$,
\end{tabular}
\end{equation}

where $\tau_{ij}$ is a smooth function defined by $\tau_{ij} \coloneqq g(A_N(e_i),e_j)$ on $U_x$.

From $(\ref{c2})$ and $(\ref{h2})$,

\begin{equation}\label{h3}
\begin{tabular}{rl}
$0=$ & $e_i(\displaystyle \sum_{j=1}^n (V^T)^j\tau_{ij})(x)$ \vspace{0.5pc} \\
~$=$ & $\displaystyle \sum_{j=1}^n\{ e_i((V^T)^j)(x)\tau_{ij}(x)+(V^T)^j(x)e_i(\tau_{ij})(x)\} $ \vspace{0.5pc} \\
~$=$ & $\kappa_i(x)+\displaystyle \sum_{j=1}^n (V^T)^j(x)e_i(\tau_{ij})(x) $.
\end{tabular}
\end{equation}

If $f(x) \neq 0$, we obtain $(V^T)^j(x) = 0 \hspace{0.1in} (1\leq j\leq n)$ from $(\ref{h2}).$ This contradicts $(\ref{h3})$ and $f(x) \neq 0$. So $f = 0$ and $X(f) = 0$ on Int$U_0$ for any vector field $X$ on $M$. From Case 1, $X(f) = 0$ on $M\backslash U_0$. This means $X(f) = 0$ on $M$. So we conclude that $f$ is constant on $M$. But $f\neq 0$ on $M\backslash U_0$, this is a contradiction. Therefore, Int$U_0 = \emptyset$.

By the same argument as in Case 1, $\kappa_i = \kappa_j$ and $X(\kappa_i) = 0$ on $M\backslash U_0$ for any $1\leq i, j\leq n$ and  any vector field $X$. From this and Int$U_0 = \emptyset$, $\kappa_i$ is constant which doesn't depend on $i$ on $M$. 

Therefore, $M$ is contained in a hypersphere.

\end{proof}



\bibliographystyle{amsbook}

\end{document}